\documentclass[11pt,a4paper]{article}
\usepackage[utf8]{inputenc}
\usepackage[T1]{fontenc}
\usepackage{amsmath}
\usepackage{amssymb}
\usepackage{amscd}
\usepackage{amsthm}
\usepackage{proof}
\usepackage{color}

\newtheorem{theorem}{Theorem}[section]
\newtheorem{corollary}[theorem]{Corollary}
\newtheorem{lemma}[theorem]{Lemma}
\newtheorem{prop}[theorem]{Proposition}
\newtheorem{rem}[theorem]{\bf Remark}
\newtheorem{definition}[theorem]{Definition}

\newtheorem{example}[theorem]{\bf Example}

\newcommand{\Hmm}[1]{\leavevmode{\marginpar{\tiny%
$\hbox to 0mm{\hspace*{-0.5mm}$\leftarrow$\hss}%
\vcenter{\vrule depth 0.1mm height 0.1mm width \the\marginparwidth}%
\hbox to 0mm{\hss$\rightarrow$\hspace*{-0.5mm}}$\\\relax\raggedright #1}}}

\newcommand{\en}{{\cal E}}
\newcommand{\dom}{\mathrm{dom}}
\newcommand{\inr}{\mathrm{Inr}}
\newcommand{\vol}{\mathrm{vol}}
\newcommand{\covr}{\mathrm{Covr}}

\newcommand{\ZZ}{\mathbb{Z}}
\newcommand{\CC}{\mathbb{C}}
\newcommand{\RR}{\mathbb{R}}
\newcommand{\NN}{\mathbb{N}}

\title{An uncertainty principle and lower bounds for the Dirichlet Laplacian on graphs}
\author{H.D.~Lenz\footnote{ Mathematisches Institut, Friedrich Schiller
Universit{\"a}t Jena, 07743 Jena, Germany, daniel.lenz@uni-jena.de},
P.R.M.~Stollmann\footnote{Fakult\"{a}t f\"{u}r Mathematik,  Technische
Universit\"{a}t Chemnitz, D-09107 Chemnitz, Germany,
stollman@math.tu-chemnitz.de }
 and G.H.~Stolz \footnote{Department of Mathematics,
University of Alabama at Birmingham,   452 Campbell Hall, Birmingham
AL 35294-1170, USA, stolz@uab.edu }
 }

\begin{document}

\maketitle

\thanks{\small\parindent0cm \textit{This paper is dedicated to W.\ Kirsch and B.\
Simon as part of the celebration of their recent birthdays. We are
grateful for their inspiration.}}

\begin{abstract}
We prove a quantitative uncertainty principle at low energies for
the Laplacian on fairly general weighted graphs with a uniform
explicit control of the constants in terms of geometric quantities.
A major step consists in establishing lower bounds for Dirichlet
eigenvalues in terms of the geometry.
\end{abstract}

\section{Introduction}

It is a phenomenon of general interest that low energy states of
Laplacians are extended in some sense. Several closely related
concepts deal with that fact. One of them is \emph{unique
continuation} for subsolutions of elliptic equations. We refer to
\cite{Ag,ABG,Aro,H,JK} for a small selection of the long list of
contributions and remark that there was renewed interest in
quantitative versions due to the importance of such results for
random Schr\"odinger operators, as seen in \cite{BK}; see also
\cite{BTV,Klein-13,KT-16,NTTV-15,RMV-13} and the literature quoted
there for more recent results. In its original form, unique
continuation means that such subsolutions cannot vanish to infinite
order. This is true in a variety of continuum contexts and certainly
not true for graph Laplacians. In fact, discrete Laplacians even
allow for eigenfunctions with compact support. For the special case
of  a tight binding model associated with the Penrose tiling the
occurrence of this effect has been known since quite some time as
witnessed for example in the physics literature
\cite{ATF,FATK,KF,KS}. This phenomenon is also interesting from a
mathematical point of view, see  \cite{DLMSchY,KLS,Ves}. For certain
planar lattices, however, unique continuation holds due to curvature
conditions as first shown  \cite{KLPS} and later generalized in
\cite{Kel}.

For the above mentioned applications to random Schr\"odinger
operators quite a different point of view is important. Namely,
states are required to be extended in the sense that the norm of
restrictions to subsets remains relevant, provided the subset one
restricts the function to is spread out in space. Many of the
references above deal with that kind of \emph{uncertainty principle}
and establish such kind of lower bounds provided the function in
question is an eigenfunction of a Schr\"odinger operator, or more
generally in the range of the spectral projection of a Schr\"odinger
operator onto a small interval of the energy axis. In view of the
above mentioned phenomenon of compactly supported eigenfunctions,
one cannot hope for an analogous result in the discrete case.
However, as we will show in the present paper, uncertainty
principles hold for low energy states of graph Laplacians and the
results allow for a uniform estimate for large classes of graphs,
with an explicit control of constants phrased in terms of geometric
properties. While our results are pretty general, we stress the fact
that they provide new insights even in the most simple cases, e.g.
the usual euclidean lattices $\mathbb{Z}^d$. For this case, related
results have been found in \cite{EK13,RM14}. For a more detailed
comparison we refer to the discussion following our main Theorem
\ref{main}.

Starting point of our method of proof is a spectral theoretic
uncertainty principle, Theorem 1.1 from \cite{BLS-10}. It deals with
a semibounded selfadjoint operator $H$ in some Hilbert space, a
bounded nonnegative operator $W$, and phrases uncertainty or unique
continuation in terms of the spectral projections $P_I=P_I(H)$ of
$H$. It says that
\begin{equation}\label{PUC}
 P_IWP_I\ge \kappa P_I
\end{equation}
provided there is $t>0$ such that
\begin{equation}\label{shift}
 \max I< \min\sigma (H+tW)=:\lambda_t
\end{equation}
Actually, in this case an explicit lower bound on $\kappa$ is easily established, viz
$$
\kappa\ge \frac{\lambda_t-\max I}{t} ,
$$
where the proof goes by contraposition, merely using the variational characterization of the
spectrum and functional calculus.

For the application we have in mind, $H$ is the Laplacian on a
weighted graph $X$ that obeys some mild assumptions and $W=1_D$ is
the indicator function of a subset that is spread out in $X$ in the
sense that for some $R>0$
$$
X\subset \bigcup_{p\in D}B_R(p) .
$$
This condition is also known as relative denseness of the set $D$ in
$X$. It is clear that in this case (\ref{PUC}) amounts to
$$
\| \phi \|^2\le \kappa^{-1}\|\phi 1_D\|^2\mbox{  for all }\phi\in {\rm Ran}(P_I)
$$
meaning that we have a quantitative unique continuation result for
linear combinations of eigenfunctions with eigenvalues in $I$ and
more general functions in the range of the corresponding spectral
projection. Here  $\kappa$ depends on $I$ and the optimal $R$ that
satisfies the above covering condition, see \cite{BLS-10}. As explained above, such
unique continuation estimates are somewhat astonishing in the graph
setting since graph Laplacians can exhibit compactly supported
eigenfunctions. However, our main result does not exclude such
compactly supported eigenfunctions and applies to graphs where the
latter occur. This is not a contradiction as our result only applies
to energy intervals $I$ concentrated near $0$.

The idea of our method can be summarized as follows. Sending
$t\to\infty$ in (\ref{shift}) with $W=1_D$ we see that the maximal
energy range for which (\ref{PUC}) gives nontrivial results is
determined by
$$
\sup_{t>0}\min \sigma(H+t1_D)\stackrel{!}{=}\min \sigma(H+\infty 1_D) ,
$$
where the non-densely defined form sum
$$
H+\infty 1_D=:H_\Omega
$$
is the Dirichlet Laplacian (in a suitable sense)  on
$\Omega:=X\setminus D$.

Our first task is therefore to get lower bounds for this $H_\Omega$
in terms of geometric quantities of the underlying graph and the
sets $\Omega$ and $D$, respectively. This is discussed  in Section
\ref{Dirichlet}.  Theorem \ref{Dir-low}  shows
$$
H_\Omega \ge \frac{1}{R\cdot \sup\{ \vol(B_R(p))\mid p\in D\}}
$$
where $R={\rm Inr}(\Omega)$ is the inradius of $\Omega$, see
Section~\ref{setup} below for the definition of volume in our
weighted graph setting. This bound is a generalization of a
well-known  bound  for finite graphs to infinite geometries under
some mild assumptions on the weighted graph.  In our proof of the
theorem, we reduce the infinite graph to a disjoint union of finite
graphs and this is a crucial step in our approach. It is achieved
 via a Voronoi type decomposition. The existence of such a
decomposition may be of interest in other contexts as well. To show
this existence   we need the rather careful analysis of basic
features of  the underlying geometry provided in Section
\ref{setup}.  Note that the bound in the theorem is weaker than what
is known in the euclidean case for $\mathbb{R}^N$, where the
corresponding Dirichlet Laplacian is bounded below by $const R^{-2}$
for domains with nice enough boundary, see Theorem 1.5.8 in Davies
\cite{Davies}. However, our bound is optimal up to constants, as  was pointed
out to us by A. Grigor'yan: Theorem 4.1 from \cite{BCG} contains examples when
this estimate is optimal for balls. Thank you, Sasha!

In order to use (\ref{PUC}) we will further need to control the
convergence of $\min \sigma(H+t1_D)$ to $\min \sigma(H_\Omega)$.
Luckily, we are in a discrete situation since the corresponding
convergence would not be true in euclidean space. In our case, under
the assumption that our Laplacian $H$ is bounded, we get convergence
in norm resolvent sense for the operators and with an explicit
convergence rate of optimal order, as shown in Section \ref{Schur}
below.

It is then easy to put things together in Section~\ref{unique} and obtain our main results,
Theorem \ref{main} and its corollary,  giving a version of
\eqref{PUC} for the case at hand with explicit control over $\kappa$
in terms of the geometry.

In Section~\ref{combgraphs} we further discuss the case of combinatorial graphs and, in particular,
compare our approach to lower bounds for $H_{\Omega}$ in Theorem~\ref{Dir-low} with the approach via
Cheeger inequalities. In Section~\ref{potentials} we discuss how Theorem~\ref{Dir-low} can
be extended to certain cases where a potential is added to $H_{\Omega}$.

Finally we mention an upcoming companion paper \cite{stolzman2} dealing with the continuum case of
Laplacians and more general divergence form operators.

\vspace{.3cm}

\section{The set-up}\label{setup}
We start by introducing our basic set-up within the context of weighted graphs, see \cite{Ke}
for a recent survey of this and related topics. A weighted graph $(X,b,m)$ is given by
\begin{itemize}
 \item a countable set $X$, finite or infinite;
 \item a symmetric weight function $b:X\times X\to [0,\infty)$ with $b(x,x)=0$ for all $x\in
 X$ and $\sum_{y\in X} b(x,y) < \infty$ for all $x\in X$;
 \item a weight function $m:X\to (0,\infty)$.
\end{itemize}
Here $m$ induces a measure on $X$ through
$$
\vol(A):=m(A):=\sum_{x\in A}m(x) .
$$
Our basic Hilbert space will be
$$
\ell^2(X,m):=\{ f\in\mathbb{C}^X\mid \|f\|^2 = \sum_{x\in X}|f(x)|^2m(x)<\infty\} .
$$
The function $b$ above should be thought of as a weight on the edges
and it appears in the energy form of the Laplacian as well as in the
distance we define on $X$. More precisely, we consider the
nonnegative form
$$
\en(f,g):= \frac12 \sum_{x,y\in X}b(x,y)(f(x)-f(y))(\overline{g(x)}-\overline{g(y)}) .
$$
We will always assume  the boundedness condition
\begin{itemize}
\item[(B)] $\sup_{x\in X} \frac{1}{m(x)}\sum_{y\in X}b(x,y)=:\delta <\infty.$
\end{itemize}
This condition  is equivalent to boundedness of the form and
consequently, the associated selfadjoint operator $H$; more
precisely, $\| H\|\le 2 \delta$. See \cite{HKLW12}, Thm.\ 9.3 and the
literature cited there. The associated selfadjoint operator is known
as \textit{weighted Laplacian} and   given by
$$
(Hf)(x)= \frac{1}{m(x)}\sum_{y\in X}b(x,y)(f(x)-f(y))
$$
for $f\in\ell^2 (X,m)$ and $x\in X$,  as follows from \cite{KL},
Theorem 5.

For us the relevant  distance on $X$ is given in the following way:
an \emph{edge} of the weighted graph $(X,b,m)$ is a set $\{ x,y\}$
with positive weight $b(x,y)>0$. Denote by $E$ the set of all edges.
Clearly, that induces the structure of a combinatorial graph
$(X,E)$. A \emph{path} is a finite sequence of edges with nonempty
intersections that can most easily be written as
$\gamma=(x_0,x_1,...,x_k)$ where $b(x_j,x_{j+1})>0$ for all
$j=0,...,k-1$; if we want to specify the endpoints we say that
$\gamma$ is a path from $x_0$ to $x_k$.  The \emph{length} of such a
path  $\gamma$ is given by
$$
L(\gamma):=\sum_{j=0,...,k-1}\frac{1}{b(x_j,x_{j-1})} .
$$
In particular the length of an edge $\{x,y\}$ is given by
$\frac{1}{b(x,y)}$. To include trivial cases we also allow trivial
paths $(x,x)$ from $x$ to $x$ whose length is $0$. We will throughout assume that our graph is
\emph{connected} in the sense that every pair of points is connected
by a path. The \emph{distance} between $x$ and $y$ is given by
$$
d(x,y):=\inf\{ L(\gamma)\mid \gamma\mbox{  a path from }x\mbox{  to
}y\} .
$$
Clearly, $d$ is symmetric and satisfies the triangle inequality.
Moreover,  by the assumptions on $b$ we have for any $x\in X$ the
estimate
$$\sup_{z} b(x,z) \leq  \sum_{z} b(x,z) < \infty$$
and this implies  $d(x,y)\geq \frac{1}{\sup_{z} b(x,z)} >0$ for any
$y\in X$ with $y\neq x$. So,  we see that $d$ separates the points
and hence we get that $d$ is a metric. We denote by
$$ U_r(x):= \{y\in X \mid d(x,y) < r\} \;\mbox{ and }\; B_r(x):= \{y\in
X \mid d(x,y) \le r\}$$ the  open and closed balls of radius $r$,
respectively.  We note in passing that, while the graph $(X,E)$ is connected,  $X$ is totally
disconnected in the topological sense as it is discrete (by what we
have just shown).

For our later considerations we will need  the Heine-Borel property
i.e.  that closed  balls in $X$ are compact. As our space has
discrete topology this is equivalent to the following finiteness
condition:
\begin{itemize}
\item[(F)] For any $x\in X$ and $r>0$ the set $B_r (x)$ is finite.
\end{itemize}
We will also need that $X$ is
 \textit{geodesic} in the sense that the following condition
holds:
\begin{itemize}
\item[(G)] Between any $x,y\in X$ there exists a path $\gamma =(x_0,x_1\ldots,
x_k)$ with $x_0 = x$, $x_k = y$ and $d(x,y) = L(\gamma)$.
\end{itemize}
It is easy  to see that (F) implies (G) (compare also  proof of part
(d) of Proposition \ref{prop-homogen} below). In fact even the
converse is true \cite{HKMW}, see Remark \ref{discussion-topology}
below as well.

By what we already mentioned in the introduction, the volume of
balls will enter our results as one important quantity. In
particular, we will need uniform bounds  on the volumes  of balls of
fixed radius:
\begin{itemize}
\item[(V)] For any $r\geq 0$ the inequality $\sup_{x\in X} m(B_r(x))<
\infty$ holds.
\end{itemize}

As $r=0$ is possible, the previous  condition clearly implies a
uniform bound on $m$ in the following form:

\begin{itemize}
\item[(M)] $m_{max}:=\sup_{x\in X} m(x) <\infty$
\end{itemize}
Given (B), it turns out that (M) alone already implies (F),   (G)
and (V). In fact, (B) and (M) together can be seen to imply a rather
homogeneous geometry. This is discussed next.  The crucial point is
that  (B) and (M) together imply a uniform upper bound for $b$ (and
even for the \textit{ vertex degree} $\mbox{deg}(x) = \sum_{y}
b(x,y)$):
$$
b_{max}:=\sup_{x,y\in X} b(x,y)\le \sup_{x\in X}\sum_{y\in
X}b(x,y)\le \delta \cdot m_{max} .
$$

\begin{prop}[Homogeneity of the geometry] \label{prop-homogen}
 Let $(X,b,m)$ be as above, in particular connected and such that
(B) and (M) hold. Then
 \begin{itemize}
\item [(a)] For any path $\gamma$ we have $L(\gamma) \ge b_{max}^{-1}\# \gamma$,
where $\#$ indicates the  cardinality. In particular, $X$ is uniform
discrete; more precisely   any two different points $x,y\in X$
  have uniform  distance at least $\frac{1}{b_{max}}$.

\item[(b)] $(X,d)$ is  locally compact and complete.
  \item[(c)] The condition (F) holds in a very uniform manner. More
  specifically,
  $$
  \# B_r(x)\le \left( r\cdot \delta \cdot m_{max}\right)^{r\cdot b_{max}}
  +1
  $$
for any $r\geq 0$ and $x\in X$.  In particular, (V) holds.

   \item[(d)]  $(X,d)$ is geodesic, i.e. (G) holds.
 \end{itemize}
\end{prop}
\begin{proof} (a) As $b(x,y) \leq b_{max}$ for all $x,y$ we see that
any two different points have minimal distance $\frac{1}{b_{max}}$.
This gives the last part of the statement of (a). Now, the first
part follows directly.

\smallskip

(b) This follows as different  points have a minimal distance by
(a).

\smallskip

(c) By (a) the points in $B_r (x)$ can be reached from $x$ by paths
with not more than $r \cdot b_{max}$ edges. Moreover, in  the
relevant paths no edge can be   longer than $r$. Thus, we will just
estimate the number of path with not more than  $r \cdot  b_{max}$
edges of length not exceeding $r$. Now, by (B)  and (M)  the number
$N_r (p)$ of edges going out from an arbitrary $p\in X$ with length
not exceeding $r$ is bounded by
$$N_r (p) \cdot  \frac{1}{r} \leq \sum_{z\in X} b(p,z) \leq \delta m_{max}.$$
The preceding considerations directly imply the given bound for
$\# B_r (x)$.   From  (M) we then obtain (V).

\smallskip

(d) By (c) any ball has only finitely many points. Consider now
arbitrary $x,y\in X$ and set $r := d(x,y)$. Then, $y$ belongs to $
B_{r+1} (x)$. By (b) the ball $B_{r+1} (x)$ has only finitely many
points. Thus, there exist only finitely many paths in $B_{r+1} (x)$
and every path from $x$ to $y$ with length less than $d(x,y) +1$
lies completely in $B_{r+1}(x)$. So the infimum over the lengths of
all paths between $x$ and $y$ can be calculated by taking the
minimum over the  lengths of paths  between $x$ and $y$ in $B_{r+1}
(x)$ and this implies (d).

This easily gives the desired statement.
\end{proof}

Our \textbf{setting for the remaining part of the paper} will be a
connected $(X,b,m)$ such that (B) and (M) hold.  By the previous
proposition, this will imply validity of (F), (G) and (V).

\smallskip

Of course, connectedness is not a real issue: if the graph is not connected it decomposes into
connected clusters and the Laplacian will just be the direct sum of the Laplacians on the
corresponding clusters. Hence our statements will remain true if properly adapted. The only change is
that $d$ as defined above is no longer a metric in the sense that the value infinity might occur.

\medskip

Although our setting allows for more general weighted graphs, readers
may always assume that we are dealing with usual combinatorial
graphs and the associated Laplacians. Our results are relevant and
new in this more specialized setting as well for which we now single
out two particularly important classes. Note that the usual
euclidean lattices belong to the first  class of examples and - up
to a multiplication of the measure by a constant  - also to the
second class of examples.

\begin{example}[Combinatorial situation] \label{combsit}
Starting from a combinatorial graph $G=(X,E)$ we set $b(x,y)=1$
whenever there is an edge from $x$ to $y$ and $b(x,y)=0$ else and
$m=1$. Then our Laplacian agrees with the usual graph Laplacian and
the distance is the well-known combinatorial or graph distance. Our
basic assumptions are satisfied if and only if  $G$ is connected and
the vertex degree is uniformly bounded.
\end{example}
\begin{example}[Normalized situation]
Let $X$ be an arbitrary countable  set with more than one element
and let $b : X\times X\longrightarrow [0,\infty)$ be symmetric with
$b(x,x) = 0$ and $\sum_{y} b(x,y) < \infty$ for all $x\in X$. Assume
that $X$ is connected and define
$$m: X\longrightarrow [0,\infty), m(x) :=\sum_{y\in X} b(x,y).$$
Due to connectedness there must exist from any $x\in X$ a $y\in X$
with $b(x,y) >0$ and we find $m(x) >0$ for any $x\in X$. As is clear
from the the construction the condition (B) holds (with $\delta =1$). In
particular,  the form $\mathcal{E}$ and the operator $H$  are
automatically bounded in this situation. So, in this case the basic
assumption is satisfied if and only if  the graph is connected and
$m$ is bounded.
\end{example}

\begin{rem}\label{discussion-topology} The metric $d$ and related
metrics are sometimes discussed under the name of \textit{path
metrics} on graphs. They  have appeared   in various  places. A
study of topological features of graphs equipped with $d$  is given
in \cite{Geo}. Completeness of the space $X$ equipped with respect
to path metrics  has played a role in recent investigations of
essential selfadjointness of Laplacians on graphs,
\cite{TH,Mil,HKMW}.  An important step in the considerations of
\cite{Mil} gives that completeness with respect to a certain path
metric implies finiteness of metric balls.   This was generalized in
\cite{HKMW} to a  Hopf-Rinow type theorem giving that for any path
metric completeness of $X$  is equivalent to finiteness of metric
balls and implies existence of geodesics. A further discussion of
$d$ and other  metrics in the context of suitable (pre)compactness
conditions for graphs is given in \cite{GHKLW}. Our framework given
by (B) and (M) and the consequences for the geometry seem not to
have been studied before.
\end{rem}

\section{Lower bounds for the Dirichlet Laplacian}\label{Dirichlet}
From the introduction we know that an interesting situation to study
is that on $D\subset X$ we have an infinite potential. Denoting by
$\Omega:=X\setminus D$ we get the form
$$
\en_\Omega(\cdot,\cdot)=\en(\cdot,\cdot)\mbox{  on
}\dom(\en_\Omega)=\{ f\in\dom(\en) = \ell^2 (X,m) \mid f=0\mbox{ on
}D\}
$$
as the limit in the strong resolvent sense of
$$
H+t1_D\mbox{  as  }t\to\infty,
$$
(see Section \ref{Schur} for further details).

 We identify $\ell^2(\Omega,m)$ with $\{ f\in\ell^2(X,m)\mid
f=0\mbox{ on }D\}$ and get an associated selfadjoint operator
$H_\Omega$ defined on  $\ell^2(\Omega,m)$. For us, $\en_\Omega$ and
$H_\Omega$ will be the restriction of the energy form and the
Laplacian, respectively, to $\Omega$ with \emph{Dirichlet boundary
conditions}.

\begin{rem}
\begin{itemize}
\item[(a)]  It is quite reasonable to call $H_\Omega$ the Dirichlet
Laplacian. In the continuum euclidean case, under mild regularity
assumptions on the boundary of $\Omega$, open in $\mathbb{R}^N$, it
holds that
  $$
-\Delta_\Omega=\lim_{t\to\infty} (-\Delta +
t1_{\mathbb{R}^N\setminus \Omega})
  $$
  is the Dirichlet Laplacian and the convergence holds in the strong resolvent sense, see Section \ref{Schur} for references.

\item[(b)] It is also reasonable to call another operator the
Dirichlet Laplacian in the discrete case, see for example
Section~5.2 in \cite{Kirsch}. This is adopted by many authors who
add a penalty term in order to force the subadditivity known from
the continuum case. In fact our Dirichlet Laplacian does in general
not  obey that $H_{U\cup V}$ is smaller than $H_U\oplus
 H_V$ for disjoint $U$ and $V$.
 \end{itemize}
\end{rem}
In our application to unique continuation the main role is played by
$D$, the set on which the potential barrier is given. For the
present section we slightly change the point of view and concentrate
on the set $\Omega\subset X$. We define the \emph{inradius} of
$\Omega$ by
$$
\inr(\Omega):=\sup\{ r>0\mid \exists x\in\Omega: U_r(x)\subset\Omega\} .
$$
We are particularly interested in
$$
\lambda_\Omega:= \min\sigma (H_\Omega) ,
$$
the bottom of the spectrum of the Dirichlet Laplacian $H_\Omega$.
Note that this latter operator is defined on $\ell^2(\Omega)$ but
its definition is always to be understood relative to the bigger
ambient graph $X$.

\smallskip

We will first deal with the finite volume situation in the next
theorem.

\begin{theorem}\label{Cheeger-fin}
Let $(X,b,m)$ be as above, in particular connected and such that (B)
and (M) hold. Let  a non-empty $\Omega\subset X$ with $\Omega \neq
X$ be given and assume  $\vol(\Omega)<\infty$ and
$\inr(\Omega)<\infty$.
\begin{itemize}
 \item [{\rm (a)}]We have
 $$
 \lambda_\Omega\ge \frac{1}{\inr(\Omega)\vol(\Omega)} .
 $$
 \item[{\rm (b)}]If  $\vol(X)<\infty$, then
 $$
 \lambda_\Omega\le \| H\|\frac{\vol(X)-\vol(\Omega)}{\vol(X)} .
 $$
\end{itemize}
\end{theorem}
\begin{proof}
Ad (a): Let $R>\inr(\Omega)$. Let $f\in\dom(\en_\Omega)$ and
$x\in\Omega$. By definition of the inradius there is $x_0\in
U_R(x)\setminus\Omega$. In particular, there is a path
$\gamma=(x_0,...,x_k)$ from $x_0$ to $x=x_k$ of length at most $R$
and $f(x_0)=0$. Therefore,
 \begin{eqnarray*}
  |f(x)|^2&=&|f(x)-f(x_0)|^2\\
  &=&|\sum_{j=0}^{k-1}\sqrt{b(x_j,x_{j+1})}(f(x_{j+1})-f(x_j))\frac{1}{\sqrt{b(x_j,x_{j+1})}}|^2\\
  &\le& \sum_{j=0}^{k-1}b(x_j,x_{j+1})|f(x_{j+1})-f(x_j)|^2\sum_{j=0}^{k-1}\frac{1}{b(x_j,x_{j+1})}\\
  &\le& \en_\Omega(f,f)R.
  \end{eqnarray*}
Since $\| f\|^2\le \sup_{x\in\Omega}|f(x)|^2\cdot\vol(\Omega)$ we get
$$
\| f\|^2 \le R\cdot\vol(\Omega)\cdot\en_\Omega(f,f)
$$
and this is the desired lower bound, since $R>\inr(\Omega)$ was arbitrary.

Ad (b): Note that $\en(1,1)=0$ under these conditions. We will
define a suitable trial function. In fact, let $\phi=c\cdot
1_\Omega$ normalized, so that $c=\vol(\Omega)^{-\frac12}$. Since $H$
and consequently $H_\Omega$ is bounded, $\phi\in\dom(\en_\Omega)$.
We will estimate the energy of $\phi$ by calculating the projection
$\phi_0=P_0\phi$, where $P_0$ is the orthogonal projection onto the
constant functions and hence leaves $H$ invariant. We get
$$
\phi_0=\frac{1}{\vol(X)}\langle \phi, 1_X\rangle 1_X\mbox{  with  }\langle \phi, 1_X\rangle =\sum_{x\in\Omega}cm(x)=\vol(\Omega)^\frac12
$$
and therefore
$$
\| \phi_0\|^2=\frac{\vol(\Omega)}{\vol(X)}\mbox{  and  }\|\phi-\phi_0\|^2=1-\|\phi_0\|^2=\frac{\vol(X)-\vol(\Omega)}{\vol(X)} .
$$
Since $H\phi_0=0$,
$$
\en(\phi,\phi)=\en(\phi-\phi_0,\phi-\phi_0)\le \|
H\|\|\phi-\phi_0\|^2=\| H\|\frac{\vol(X)-\vol(\Omega)}{\vol(X)}.
$$
As $\phi$ is supported in $\Omega$ we have  $\en_\Omega (\phi,\phi)
= \en (\phi,\phi)$ by the definition of $\en_\Omega$ and the
preceding estimate gives  $\lambda_\Omega$ of $H_\Omega$.
\end{proof}
\begin{rem} \label{ChuDavies}
 \begin{enumerate}
  \item[(a)]
For finite combinatorial graphs, the lower bound is a familiar bound
and our proof follows known lines, compare Lemma 1.9 in \cite{Chu} and Lemma 2.4 in \cite{BCG} for
related estimates.

\item[(b)]
The upper bound is interesting as it shows that a lower bound like
in the continuum euclidean case for domains with nice enough boundary, namely in the form
$const\cdot\inr(\Omega)^{-2}$, see Theorem 1.5.8 in \cite{Davies},
will not be possible! Indeed, for any connected  finite graph $X$ we
can take $D$ to consist of just a single element of $X$. Then,
$$\lambda_\Omega \leq \|H\|\frac{\vol(X)-\vol(\Omega)}{\vol(X)} \leq  \|H\|
\frac{m_{max}}{\vol(\Omega)}$$ will be bounded in terms of the
inverse volume of $\Omega$. Now, this can be  much smaller than a
second power of the inverse inner radius as can be seen by considering e.g.\
a ball in an $N$-dimensional  euclidean lattice with $N\geq 3$  and
choosing as set $D$ just the center of this ball.

\item[(c)] As mentioned in the introduction, Theorem 4.1 from \cite{BCG} gives that
the lower bound is optimal up to constants.
  \end{enumerate}
\end{rem}

To lift the above result to the case of infinite volume, we
introduce the concept of a Voronoi decomposition. This concept
may be of interest in other contexts as well. In fact, for
combinatorial Laplacians it has already proven useful in
\cite{SSV-14}.
\begin{definition}\label{def-vor}
Let $(X,b,m)$ be as above and  $D\subset X$ non-empty. A
\emph{Voronoi decomposition} of $X$ with centers from $D$ is a
pairwise disjoint family $(V_p)_{p\in D}$ such that  following
conditions hold:
 \begin{enumerate}
  \item[(V1)] For each $p\in D$ the point $p$ belongs to  $ V_p$  and
  for all $x\in V_p$ there exists a path $\gamma$ from $p$ to $x$ that lies
   in $V_p$ and satisfies  $L(\gamma) = d(p,x)$.
  \item[(V2)] For each $p\in D$ and for all  $x\in V_p$ the inequality
  $d(p,x) \leq  d(q,x)$ holds  for any $q\in D$.
  \item[(V3)] $\bigcup_{p\in D}V_p=X$.
 \end{enumerate}
\end{definition}

\begin{rem} The condition (V1) and (V2)  imply that for any $p\in D$
\begin{itemize}
\item the set   $V_p$ contains $p$ and is connected and
  \item  any $x\in V_p$ satisfies $d(p,x)\leq d(q,x)$ for any $q\in D$.
\end{itemize}
However, it is not hard to see by examples that (V1) and (V2) are
even stronger than these two conditions, i.e.\ that connectedness of the $V_p$ does not imply that they contain geodesics.
\end{rem}

In our investigation of Voronoi decompositions, we will need some
further concepts. We define the \textit{covering radius} of $D$ by
$$
\covr(D):=\inf\{ R>0\mid \bigcup_{p\in D}B_R(p)=X \} \in [0;\infty] ,
$$
with the usual convention $\inf\emptyset =\infty$ and say that $D$
is \emph{relatively dense}, provided $\covr(D) < \infty$.

\begin{lemma}\label{lem-covr-equal-inr}
 Let $(X,b,m)$ be as above, $D\subset X$ and $\Omega:=X\setminus D$. Then
 $$\covr(D)=\inr(\Omega) .$$
\end{lemma}
\begin{proof}
Let $R<\covr(D)$ (which is set to $\infty$ if $D$ is not relatively
dense). Then $ \bigcup_{p\in D}B_R(p)\neq X$ which means that there
is $x_0\in\Omega$ with $B_R(x_0)\cap D=\emptyset$ and therefore
$\inr(\Omega)\ge R$. Consequently, $\covr(D)\le \inr(\Omega)$.

Conversely, $R<\inr(\Omega)$ gives $x_0\in\Omega$ and $R<
\widetilde{R} < \inr (\Omega)$
 s.t. $B_R(x_0) \subset U_{\widetilde{R}} (x_0) \subset\Omega$ which means that $x_0\not\in
\bigcup_{p\in D}B_R(p)$ and, therefore, $R<\covr(D)$. Consequently,
$\covr(D)\ge \inr(\Omega)$.
\end{proof}

\begin{rem} If $D$ is relatively dense in $X$ then  the infimum in
the definition of the covering radius is even a minimum i.e. $X =
\bigcup_{p\in D} B_R (p)$ for $R = \covr(D)$ holds. To see this
chose an arbitrary $x\in X$ and consider $B_{R+1} (x)\cap D$.   By
the definition of the covering radius this set  contains a sequence
$(p_n) \subset D$ with $\inf d(p_n,x)\leq  R$. Moreover, by (F) this
set is finite. Thus, it must contain a $p\in D$ with $d(p,x)\leq R$.
As $x\in X$ was arbitrary the desired statement follows.
\end{rem}

Here is our result on existence of a Voronoi decomposition.

\begin{prop}\label{prop-existencevoronoi}
Let $(X,b,m)$ be as above and assume that  $D\subset X$ is
non-empty. Then there exists a Voronoi decomposition with
centers from $D$. Moreover, whenever $R=\covr(D)$ is finite then any
Voronoi decomposition $(V_p)_{p\in D}$ of $X$ with centers from
$D$ has the property that  $V_p\subset B_R(p)$ for all  $p\in D$.
\end{prop}

\begin{proof} We first show existence of a Voronoi decomposition
with centers from $D$. A family $(V_p)_{p\in D}$  of pairwise
disjoint subsets of $X$ is called admissible, if it satisfies (V1)
and (V2) from Definition \ref{def-vor} above. Evidently, $V_p=\{
p\}$, $p\in D$,  gives such an admissible family. With the obvious
ordering we can apply Zorn's lemma and get a maximal admissible
family. We will show now that such a maximal family is a Voronoi
decomposition, i.e., satisfies as well (V3):
 $$
\bigcup_{p\in D}V_p=X.
$$
Assume otherwise. Then there exists an  $x\in X$ which does not
belong to
$$W:= \bigcup_{p\in D} V_p.$$
Now as $X$ is connected and $D$ is not empty, there exists an $R>0$
such that  the set
$$S:= B_R(x)\cap D$$ is not empty. Indeed, we may just take $R =
d(x,q)$ for any $q\in D$.  Moreover,  $S$ is finite as $B_R (x)$ is
finite (due to Proposition \ref{prop-homogen}). Therefore, there
exists a $p\in S$ with minimal distance to $x$ i.e. with
\begin{equation}\label{minimal}
d(p,x)\leq d(u,x)
\end{equation} for any $u\in S$. By  $p\in S$,
clearly, $d(p,x) \leq R$ holds. Thus, \eqref{minimal} holds also for
$u \in (X\setminus B_R (x))\cap D$. Hence, we see that
\eqref{minimal} holds for all $u\in D$.

Moreover, as our space is geodesic due to Proposition
\ref{prop-homogen}, there exists a  path $\gamma=(x_0,...,x_k)$ with
$x_0=p$ and $x_k=x$ and
$$d(p,x) = L(\gamma)  = \sum_{j=0}^{k-1} b(x_j, x_{j+1})^{-1}.$$
Then, for any $u\in D$ we must have
\begin{equation}\label{folg}
d(u,x_l) \geq \sum_{j=0}^{l-1} b(x_j, x_{j+1})^{-1}
\end{equation}
for any $l=1,...,k$, as otherwise we would arrive at
$$d(u,x) \leq d(u,x_{l}) + d(x_{l}, x) < \sum_{j=0}^{k-1} b(x_j,
x_{j+1})^{-1} = d(p,x)$$ which contradicts \eqref{minimal}. Consider now
the smallest index $l\in \{0,\ldots, k-1\}$ with $x_l \in W$ and
$x_{l+1}\notin W$. (Such an $l$ exists as $x_0 = p\in W$ and $x =
x_k\notin W$ by our assumption.) Let $q\in D$ be such that $x_l \in
V_q$. By (V2) (applied to $V_q$) we then have
$$d(q,x_l) \leq d(p,x_l) \leq \sum_{j=0}^{l-1} b(x_j, x_{j+1})^{-1}.$$
Combined with \eqref{folg} this gives
\begin{equation}\label{e-eins} d(q,x_l) = d(p,x_l) = \sum_{j=0}^{l-1} b (x_j,
x_{j+1})^{-1}.
\end{equation}
Putting this together we arrive at
\begin{eqnarray*}
d(q,x_{l+1}) &\leq & d(q,x_l) + d (x_l, x_{l+1})\\
& \leq & d(q, x_l) +  b(x_l, x_{l+1})^{-1}\\
\eqref{e-eins} &=&\sum_{j=0}^{l} b (x_j, x_{j+1})^{-1} \\
\eqref{folg} &\leq & d(u,x_{l+1})
\end{eqnarray*}
for any $u\in D$. This  chain of inequalities   gives not only
$$d(q,x_{l+1}) \leq d(u,x_{l+1})$$
for all $u\in D$ but also (if we set $u =q$)
$$ d(q, x_{l+1}) = d(q, x_l) + b(x_{l}, x_{l+1})^{-1}.$$
Thus, we could add $x_{l+1}$ to $V_q$ and obtain the admissible
decomposition $(\widetilde{V}_u)_{u\in D}$ with $\widetilde{V}_q :=
V_q \cup \{x_{l+1}\}$ and $\widetilde{V}_u = V_u$ for $q\neq u \in
D$. This is a contradiction to maximality. Thus, we infer that a
maximal admissible family satisfies (V1), (V2) and (V3).

\medskip

We now show the last statement. So, let $(V_p)$, $p\in D$, be a
Voronoi decomposition. Let $p\in D$ and $x\in V_p$ be arbitrary.
As
 $\covr(D) =R<\infty$ there must exist a $q\in D$ with $d(q,x)\leq R$.  By
 (V2), we then infer
 $$d(p,x) \leq d(q,x) \leq R.$$
 This shows  $V_p\subset B_R (p)$.
\end{proof}

\begin{rem} We note that the proof of the previous proposition does
not require (B) and (M) but only the weaker  Heine-Borel property
(F) (and the resulting existence of geodesics (G)). Thus, the
proposition will be true in even more general situations than the
standard setting of our paper. Note also that the existence
statement of the proposition does not need relative denseness of the
set $D$. Although the existence of a Voronoi decomposition might sound
rather natural, it does require some non-trivial geometric input. It does not
hold under the sole assumption of (G), as can be seen by the following simple example.
\end{rem}

Here comes a geodesic weighted graph that does not allow a Voronoi decomposition:

\begin{example}
Let $X:=\left(\NN\times \{ 0\}\right)\cup \{ (1,1)\}$ with weight $b((n,0);(n+1,0))=2$ for $n\in\NN$,
$b((n,0),(1,1))=(1+\frac{1}{n})^{-1}$ for $n\in\NN$ and $b(x,y)=0$ else.
Since none of the points from
$D:=\NN\times \{ 0\}$ is closest to the point $(1,1)$, there is no Voronoi decomposition of $X$ with
centers in $D$ in the above sense.
\end{example}

 The next example is a variant of the previous one with
the additional feature that  each vertex has at most three  adjacent
vertices. It consists of one ray with finite total diameter and one
additional leave emanating at each site. The edge weights are chosen
such that the distance of the leaves to the 'origin' of the ray
become shorter and shorter. Here, are the details:

\begin{example} Let $X =
(\NN\times\{1\}) \cup (\NN\times \{0\})$. Set
$$b ((n,1), (n+1,1)) := {2 \cdot 4^{n-1}}$$
and
$$b ((n,1),(n,0)) := {4^{n-1}}$$
and $b(x,y) =0$ else. The example is a tree and, hence, clearly
satisfies (G). Moreover,  a short computation shows that the
distance of $(n,0)$ to $(1,1)$ is given by $\frac{2}{3} +
\frac{1}{3} \frac{1}{4^{n-1}}$. So, none of the points from
$D:=\NN\times \{0\}$ is closest to $(1,1)$. Hence, there is no
Voronoi decomposition.
\end{example}

\begin{theorem}\label{Dir-low}
Let $(X,b,m)$ be as above, in particular connected and such that (B)
and (M) hold and assume that $D\subset X$ is relatively dense,
$\Omega:=X\setminus D$. Then,
  $$
  \lambda_\Omega\ge\frac{1}{\inr(\Omega)\cdot \vol [\inr(\Omega)]},$$
  \mbox{  where  }
  $$\vol[s]:= \sup_{x\in X}\vol(B_s (x)) .
  $$
\end{theorem}
\begin{rem} (a)  Note that $\vol[s] <\infty$ for any $s\geq 0$  due to
Proposition \ref{prop-homogen}.

(b) The proof  shows that we can actually replace $\vol [s]$ by the
slightly better $\vol_\Omega [s] :=\sup_{x\in X}\vol(B_s (x)\cap
\Omega)$.
\end{rem}

\begin{proof}
Let  $V_p$, $p\in D$, be the Voronoi decomposition from the
preceding proposition. Then, for $f\in\dom(\en_\Omega)$ (i.e., $f=0$ on $D=X\setminus \Omega$), we have
$$\mathcal{E}(f,f) = \frac{1}{2}\sum_{x,y} b(x,y) (f(x) - f(y))^2 \geq \frac{1}{2}\sum_{p\in D} \sum_{x,y\in V_p} b(x,y)(f(x)- f(y))^2.$$
For each $p\in D$ we have $f\in\dom(\en_{V_p \setminus \{p\}})$ and thus
$$ \frac{1}{2} \sum_{x,y\in V_p} b(x,y)(f(x)- f(y))^2 \geq \lambda_{V_p\setminus \{p\}} \|f 1_{V_p\setminus \{p\}}\|^2 = \lambda_{V_p\setminus \{p\}} \|f 1_{V_p}\|^2.$$
To this we apply Theorem~\ref{Cheeger-fin} with $\Omega= V_p \setminus \{p\}$ to get, after reinserting the sum over $p\in D$,
$$\mathcal{E}(f,f) \geq \sum_{p\in D} \frac{1}{\inr (V_p \setminus\{p\})
\vol (V_p \setminus \{p\})} \|f 1_{V_p}\|^2 .$$  By
the previous proposition we know $V_p \subset B_R (p)$ with $R =
\covr (D)$. Moreover, by (V1) any $x\in V_p$ is connected to $p$  by
a path in $V_p$ of length $d(p,x)\leq R$. This gives $\inr(V_p
\setminus \{p\})\leq R$.  As $R= \covr(D) =\inr (\Omega)$ (see Lemma
\ref{lem-covr-equal-inr}) the desired statement follows easily.
\end{proof}

While in the context of our subject matter here the framework of bounded Laplacians and more specifically
the assumptions (B) and (M) are natural requirements, the study of lower bounds of Laplacians can be
generalized to more general weighted graphs. We refer to the upcoming papers \cite{lowerbounds} and \cite{poincare}
for more details.
\section{The large coupling limit}\label{Schur}

In this section we study the large coupling limit
$$
H+ t 1_D \rightarrow H_{\Omega}
$$
more thoroughly. Recall that in the situation we have in mind, $H$
is a weighted Laplacian on some weighted $\ell^2(X,m)$ and $D
\subset X$ is a subset, $\Omega = X \setminus D$. We already noticed
that the above convergence takes place in the strong resolvent
sense, as can easily be seen from Kato's monotone convergence
theorem, see \cite{Kato}, Thm 3.13a, p.461 for the densely defined
and \cite{Simon-78}, Thm. 4.1, p. 383 for the general case. Under
the assumption that $H$ is bounded, we actually see that we even
have norm resolvent convergence with optimal decay rate $1/t$ and
uniform bounds that depend on $\|H\|$ only, see Proposition
\ref{Schurprop} below.

The analogous problem is much more intricate in the continuum.
There, with $H=-\Delta$, one needs certain regularity assumptions on
$\Omega$ to even get strong resolvent convergence and these
regularity assumptions will not suffice to decide norm resolvent
convergence. In this context we refer to  \cite{BAB,BD} for results
on norm convergence in a rather general framework and further
references.

In our case, we don't even need to take into account the special
structure, so in what follows, let $\cal{H}$ be a Hilbert space, $P:
{\cal H} \to {\cal H}_1$ the orthogonal projection onto a closed
subspace $\{0\} \not= {\cal H}_1 \subsetneq {\cal H}$ (to avoid
trivialities), $Q$ the orthogonal projection onto ${\cal
H}_1^{\perp} =: {\cal H}_2$, $0\le H$ a bounded selfadjoint operator
on ${\cal H}$ and
$$
H_t := H + t Q^*Q.
$$
Again, monotone convergence implies that the corresponding forms
$\en_t$ converge, as $t\to \infty$, to the closed form
$\en_{\infty}$ given by
$$
\dom(\en_{\infty}) = \{ f\in {\cal H} \mid \sup_t \,\langle H_t f, f \rangle < \infty\} = {\cal H}_1 \
$$
$$
\en_{\infty}(f,g)  =  \langle Hf, g \rangle
$$

\begin{rem}
The unique selfadjoint operator in ${\cal H}_1$ associated with
$\en_{\infty}$ is given by $PHP^*$.
\end{rem}

\begin{prop} \label{Schurprop}
For $t\geq 2 \|H + 1\|^2$ we have
$$
\|(H_t+1)^{-1} - P^* (PHP^*+1)^{-1}P\| \le \frac{4\|H+1\|^2}{1+t}.
$$
\end{prop}

Note that $P^*(PHP^*+1)^{-1} P = (PHP^*+1)^{-1} \oplus 0$, the resolvent of $PHP^*$ which is defined on ${\cal H}_1$, extended by $0$ to ${\cal H}_1^{\perp} = {\cal H}_2$.

\begin{proof}[Proof of the Proposition]
We use the Schur complement by decomposing $H_t$ according to ${\cal H}_1 \oplus {\cal H}_2$ into a block operator matrix
$$
H_t  = \left( \begin{array}{cc} P H_t P^* & P H_t Q^* \\ Q H_t P^* & Q H_t Q^* \end{array} \right).
$$
Since $PQ^* = QP^* =0$,
$$
H_t + 1 = \left( \begin{array}{cc} P(H+1)P^*   & P(H+1)Q^* \\
Q(H+1)P^* & Q(H+t+1)Q^* \end{array} \right) =: \left(
\begin{array}{cc} A & B
\\ B^* & D_t \end{array} \right).
$$

For $t\ge 0$, $D_t$ is invertible with $\|D_t^{-1}\| \le (1+t)^{-1}$. Consequently, the Schur complement
$$
S_t := P(H+1)P^* - BD_t^{-1} B^*
$$
is boundedly invertible for large enough $t$, since
$$
P(H+1)P^* \ge 1 \quad \mbox{and} \quad \|BD_t^{-1} B^*\| \le \|H+1\|^2 \frac{1}{1+t}.
$$
More precisely, we get
$$
S_t \ge \frac12\mbox{  and }\| S_t^{-1}\|\le 2\mbox{  for  }t\ge 2 \|H+1\|^2 .
$$
Using the Schur complement to invert $H_t+1$ gives
$$
(H_t+1)^{-1} = \left( \begin{array}{cc} S_t^{-1} & -S_t^{-1} B D_t^{-1} \\ - D_t^{-1} B^* S_t^{-1} & D_t^{-1} (1+B^* S_t^{-1} B)D_t^{-1} \end{array} \right),
$$
Therefore,
\begin{eqnarray*}
\lefteqn{(H_t+1)^{-1} -P^* (PHP^*+1)^{-1}P} \\ & = & \left( \begin{array}{cc} S_t^{-1}- (P(H+1)P^*)^{-1}& -S_t^{-1} B D_t^{-1} \\ - D_t^{-1} B^* S_t^{-1} & D_t^{-1} (1+B^* S_t^{-1} B)D_t^{-1} \end{array} \right)
\end{eqnarray*}
can be bounded in norm by
$$
2 \max\{\| S_t^{-1}- (P(H+1)P^*)^{-1}\|,\| D_t^{-1} B^* S_t^{-1}\|,
\| D_t^{-1} (1+B^* S_t^{-1} B)D_t^{-1}\|\} .
$$
Using the resolvent equation for the first term as well as the above
bounds gives the claim.
\end{proof}

\begin{lemma}
Let $H_2 \ge 0$ be a selfadjoint operator on  ${\cal H}$, $H_1\ge 0$
a selfadjoint operator (possibly on a subspace  ${\cal H}_1$),
$\lambda_i := \min \sigma(H_i)$ for $i=1,2$ and assume that
$\lambda_1\ge\lambda_2$. Then
\begin{eqnarray*}
0\le \lambda_1- \lambda_2 &\le& (\lambda_1+1)^2 \|(H_1+1)^{-1} - (H_2+1)^{-1}\| \\
&\le&
 \| H_1+1\|^2 \|(H_1+1)^{-1} - (H_2+1)^{-1}\| ,
\end{eqnarray*}
the latter provided $H_1$ is bounded.
\end{lemma}

\begin{proof}
Denote $\delta := \|(H_1+1)^{-1} - (H_2+1)^{-1}\|$. Since $H_i \ge
\lambda_i$ we get $(H_i+1)^{-1} \le (\lambda_i+1)^{-1}$ which gives
$(H_2+1)^{-1} \le (\lambda_1+1)^{-1} + \delta$. Thus
$$
H_2 +1 \ge \frac{1}{\frac{1}{\lambda_1+1}+ \delta} =
\frac{\lambda_1+1}{1+\delta(\lambda_1+1)}
$$
and
$$
\lambda_2 \ge
\frac{\lambda_1+1-1-\delta(\lambda_1+1)}{1+\delta(\lambda_1+1)} =
\lambda_1 - \delta \left( \frac{(\lambda_1+1)^2
}{1+\delta(\lambda_1+1)} \right) \ge \lambda_1 - (\lambda_1+1)^2
\delta,
$$
as claimed. In case that  $H_1$ is bounded, the spectrum is bounded
by the norm. This argument extends to the case where $H_1$ is defined on a subspace ${\cal H}_1$, with $(H_1+1)^{-1}$ is to be read as $(H_1+1)^{-1} \oplus 0$ on ${\cal H} = {\cal H}_1 \oplus {\cal H}_1^{\perp}$.
\end{proof}

Given the previous two results we immediately infer the following.

\begin{corollary}\label{cor-schur}
In the situation of Proposition~\ref{Schurprop}, let
$$
\lambda_t:= \min \sigma(H_t)\mbox{  and  }\lambda_\infty:= \min \sigma(PHP^*) .
$$
Then, we have $\lambda_\infty \geq \lambda_t$ for all $t\geq 0$ and
 for $t\ge 2 \|H + 1\|^2$,
\begin{eqnarray*}
 \lambda_t &\ge& \lambda_\infty - \frac{4\|H+1\|^2
(\lambda_\infty+1)^2}{t+1}\geq  \lambda_\infty -
\frac{4\|H+1\|^4}{t+1} .
 \end{eqnarray*}
\end{corollary}

\section{A quantitative unique continuation result for the Laplacian on graphs}\label{unique}

We are now in position to derive our main result by combining what
we have established so far. We let $(X,b,m)$ be as above, in
particular connected and such that (B) and (M) hold and assume that
$D\subset X$ is relatively dense, $\Omega:=X\setminus D$.

We recall that by \cite{BLS-10}, Thm 1.1,

$$
 P_I(H) 1_D P_I(H)\ge \kappa P_I(H)
$$
provided there is $t>0$ such that
\begin{equation}
 \max I< \min\sigma (H+t1_D)=:\lambda_t
\end{equation}
Actually, in this case
$$
\kappa\ge \frac{\lambda_t-\max I}{t} .
$$
Moreover, in virtue of Theorem \ref{Dir-low}
 we have a lower bound on $\lambda_\Omega = \min \sigma
(H_\Omega)$ viz
$$
\lambda_\Omega \ge \frac{1}{\inr(\Omega)\cdot \vol [\inr(\Omega)]}
$$
 and, by Corollary \ref{cor-schur}
applied with $\lambda_\infty = \lambda_\Omega$,  we have  a lower
bound
\begin{equation}\label{lowbd}
\lambda_t \ge \lambda_\Omega -
4\|H+1\|^2 (\lambda_\Omega+1)^2\frac{1}{t+1}
\end{equation}
for $t\ge 2 \|H+1\|^2$.  From
these bounds we easily get
\begin{theorem}
 \label{main} Let $(X,b,m)$ be as above, $D\subset X$ relatively dense and $\Omega:=X\setminus D\not=\emptyset$. Let $I\subset\mathbb{R}$ such that $\max I <\lambda_\Omega$.
 Then
 \begin{equation}\label{mainineq}
  P_I(H) 1_DP_I(H) \ge \frac{(\lambda_\Omega-\max I)^2 }{16\| H+1\|^2(\lambda_\Omega+1)^2} P_I (H).
 \end{equation}
\end{theorem}
\begin{proof}
Using (\ref{lowbd}) we find
 $$
 P_I(H) 1_DP_I(H) \ge \kappa P_I(H)
$$
for a $\kappa$ satisfying
$$
\kappa \ge  \sup_{t\ge 2\|H+1\|^2 }\frac{\lambda_\Omega -\max I
-4\|H+1\|^2 (\lambda_\Omega+1)^2\frac{1}{t+1}}{t}.
$$
To make things easier we replace $\frac{1}{t+1}$ by $\frac{1}{t}$ which
gives a lower bound; hence, we are left to find the maximum of
$$
f(t):= \frac{c_0}{t} - \frac{c_1}{t^2}  \mbox{  for  }t\ge 2\|H+1\|^2
$$
for an appropriate choice of $c_0$ and $c_1$. The corresponding argument is
$$
t_{max}=2\frac{c_1}{c_0}=8\frac{\|H+1\|^2 (\lambda_\Omega+1)^2}{\lambda_\Omega
-\max I}\ge 2\|H+1\|^2
$$
with maximal value
$$
\frac{(\lambda_\Omega-\max I)^2 }{16 \|H+1\|^2 (\lambda_\Omega+1)^2}
,
$$
the assertion.
\end{proof}

Given the preceding theorem we can now  use the lower bound from
Theorem \ref{Dir-low} and the trivial bound  $\lambda_\Omega\le\|
H_\Omega\|\le \| H\|$ to obtain the following corollary.

\begin{corollary}\label{coro-main}
Let $(X,b,m)$ be as above, $D\subset X$ relatively dense and
$\Omega:=X\setminus D\not=\emptyset$. Let $I\subset\mathbb{R}$ such
that $\max I < \frac{1}{\inr(\Omega)\cdot \vol [\inr(\Omega)]}$.
Then
 \begin{equation}\label{corineq}
 P_I(H) 1_D P_I(H) \ge \frac{(\frac{1}{\inr(\Omega)\cdot \vol
[\inr(\Omega)]}-\max I)^2 }{16\| H+1\|^4} P_I (H) .
 \end{equation}
\end{corollary}

We go on to compare our results with those obtained by Rojas-Molina
\cite{RM14} and Elgart and Klein \cite{EK13} who treat the euclidean
lattices $\ZZ^d$. First note that distances are measured somewhat
differently in comparison to what we do here, e.g., an $R$-Delone
set in the sense of  \cite{RM14} would have covering radius bounded
above by $d\cdot R$. Therefore, we suppress such numerical constants
right now to make things easier, and use $\gtrsim$ to indicate the
corresponding relations.

The energy range for the uncertainty principle (in our notation the
largest possible $\max I$)  is given by $\tilde{E}_W\gtrsim R^{-2d
-2}$ in \cite{RM14}  while our result gives the better
$\lambda_\Omega\gtrsim R^{-d -1}$.
In comparison, \cite{EK13}, provide a lower bound
$\lambda_\Omega\gtrsim   R^{-2d}$ using a Cheeger inequality. Here,
again our estimate is better. We discuss this in more detail in the next section.

\section{A closer look at the combinatorial situation} \label{combgraphs}

In this section we consider the combinatorial situation and exhibit
a large class of models, viz combinatorial graphs with
subexponential growth of balls, to which our results can be applied.
Along the way we also compare our approach to the approach via
Cheeger inequalities used in \cite{EK13}.

\medskip

Throughout this section we consider the case of combinatorial
graphs, i.e. a connected graph $(X,b,m )$ with $m\equiv 1$ and $b$
taking values in $\{0,1\}$  such that with a suitable $\delta\geq 0$
we have $\sum_{y\in X} b(x,y) \leq \delta$ for all $x\in D$. In
fact, in this case $\delta := \sup_x \#\{ y\in X: b(x,y)=1\}$ is the
maximal vertex degree and
$$\vol[s]=\sup_{x\in X} \vol (B_s (x)) = \sup_{x\in X} \# B_s(x).$$

\smallskip

We will be particularly interested in the case  $\inf \sigma (H)
=0$, as we will have non-trivial applications of our main results in
this case.  As is well known, this case  can be characterized via
the \textit{Cheeger constant} or \textit{isoperimetric constant}
$$\beta := \inf_{\emptyset \neq S  \subset X, \# S<\infty}
\frac{ \# \partial S}{\vol(S)},$$ where the
\textit{combinatorial boundary} $\partial S$ of $S$ is given by
$$\partial S:=\{ (x,y) \in X\times X : x\in S, y\notin S \mbox{ and } b(x,y) =1 \}.$$
Indeed, this characterization is given as
$$\inf \sigma (H) = 0 \Longleftrightarrow \beta =0.$$
Here, the  implication `$\Longleftarrow$'  follows easily by a
direct computation. Specifically, for any finite set $S$ we find
$$\en (1_S,1_S) \leq \# \partial S$$
as well as $\|1_S\|^2  = \vol (S)$,   where $1_S$ denotes the
characteristic function. The implication `$\Longrightarrow$' follows
from the Cheeger inequality
$$\inf \sigma (H) \geq \frac{\beta^2}{2\delta}.$$
This inequality goes back to \cite{Dod} (in a slightly different
formulation), see \cite{BKW} and the discussion below as well. A
well-known consequence of these considerations is that   $\inf
\sigma (H) = 0$ whenever the volume growth  of balls is
subexponential:

\begin{prop} Assume that for one (and thus all) $x\in X$
\begin{equation}\label{subexp}
\limsup_{n\to \infty} \frac{\log \vol(B_n (x))}{n} = 0.
\end{equation}
Then, $\inf \sigma (H) = 0$.
\end{prop}
\begin{proof} The assumption implies that
$$\inf_n \frac{ \vol ( B_{n+1}(x) \setminus B_n(x))}{\vol (B_n (x))}
= 0$$ (as otherwise we had $\vol(B_{n+1} (x)) \geq (1 + \alpha)^{n}
\vol (B_1 (x))$  with $\alpha$ being the non-vanishing value of the
infimum and this would lead to exponential volume growth). Moreover,
a direct combinatorial argument shows that
$$\# \partial B_n (x) \leq \delta \cdot \vol( B_{n+1}(x) \setminus B_n(x) ) .$$
Putting this together we infer $\beta = 0$.  Hence, by the preceding
considerations the statement on the infimum of the spectrum follows.
\end{proof}

Now let  $D\subset X$ be relatively dense and set
$\Omega:=X\setminus D\not=\emptyset. $ If $\beta =0$ (or,
equivalently, $\inf \sigma (H) =0$), we obtain from Theorem
\ref{Dir-low}
$$\lambda_\Omega = \inf \sigma (H_\Omega) \geq \frac{1}{\inr(\Omega)\cdot \vol
[\inr(\Omega)]}> 0 = \inf \sigma (H).$$ So, the infimum of the
spectrum of $H_\Omega$ is indeed bigger than the infimum of the
spectrum of $H$.  In this case, whenever $I\subset\mathbb{R}$ is an
interval containing $0$ with $\max I < \frac{1}{\inr(\Omega)\cdot
\vol [\inr(\Omega)]}$, we have $P_I (H) \neq 0$. So, we can in
particular  apply  from Corollary \ref{coro-main} to obtain a
non-trivial inequality.
As $\beta =0$ holds whenever \eqref{subexp} is satisfied, these
considerations  provide a large class of examples in which our
approach can be carried out.

\smallskip

If $\beta =0$ and  $D\subset X$ is relatively dense it is also
possible to obtain a lower bound for $\inf \sigma (H_\Omega)$ via a
Cheeger type inequality. Specifically, we define
$$\beta_\Omega:=\inf_{\emptyset \neq S \subset \Omega, \# S<\infty }
\frac{ \# \partial S}{\vol(S)}.$$

For the case of Euclidean lattices this is carried out  in
\cite{EK13}. The general case can be inferred from \cite{Dod,BKW},
as we will explain below. To apply this in a meaningful way we need
an explicit estimate on $\beta_\Omega$. The case  $X = \ZZ^d$  and
$D\subset X$ relatively dense is treated in \cite{EK13} and it is
shown that $\beta_\Omega \geq \ C / R^{d}$ with a suitable constant
$C$ and   $R$ being the covering radius of $D$. It turns out that a
similar bound can be obtained in the general case as well. In fact,
this can be shown rather directly based on the Voronoi
decomposition provided in Proposition \ref{prop-existencevoronoi}.

\begin{prop} Let $D\subset X$ be relatively dense with covering
radius $R$ and set $\Omega = X\setminus D$. Then,
$$\beta_\Omega \geq \frac{1}{\vol[R]}.$$
\end{prop}
\begin{proof} From Proposition \ref{prop-existencevoronoi} we obtain
a Voronoi decomposition $(V_p)_{p\in D}$ with centers in $D$.
Let $S\subset \Omega$ be an arbitrary non-empty  finite subset of
$\Omega$.  Consider now an arbitrary $p\in D$ with $V_p \cap S\neq
\emptyset$. Then, $V_p$ must contain $u\in X\setminus S$  and a
$w\in S$ with $b(u,w) = 1$. (To see this it suffices to consider a
path in $V_p$ from a $q\in S\cap V_p \neq \emptyset$ to $p\in D
\subset  X\setminus S$. Such a path exists by  (V1). As such a path
starts in $S$ and finishes in the complement of $S$, there must
exist a first edge,  where it leaves $S$. This edge gives the
desired points $u,w$.) So, any $V_p$ that intersects $S$ provides at
least one 'boundary edge'. Consequently, with
$$N(S):=\#\{p\in D: V_p \cap S\neq \emptyset\}$$
we have
$$\# \partial  S\geq N (S).$$
At the same time we also clearly have
$$\vol (S) \leq N(S)\sup_{p\in D}  \vol (V_p) \leq N(S) \vol [R],$$
where we use that any $V_p$ is contained in a ball of radius $R$.
Putting the last two estimates together we find $\# \partial S
/ \vol (S) \geq\frac{1}{\vol[R]}$. As $S$ was an arbitrary non-empty
finite subset of $X$  the desired estimate on $\beta$ follows.
\end{proof}

Based on this proposition and the Cheeger inequality from \cite{BKW}
we obtain for $D\subset X$ with $R = \covr(D)$ the lower bound

\begin{equation}
 \label{alphacheeger}
\lambda_\Omega \geq \frac{\beta_\Omega^2}{2\delta} \geq
\frac{1}{2\delta \cdot \vol [R]^2}.
\end{equation}

In fact, in \cite{BKW}, the main point is to deal with a different
isoperimetric constant $\alpha$, defined in terms of an intrinsic
metric $\rho$ that allows for a Cheeger inequality in the case of
unbounded vertex degree. In our simpler situation, we can choose as
an intrinsic metric the following multiple of the combinatorial
metric, namely
$$
\rho(\cdot, \cdot):= \delta^{-\frac12} d(\cdot,\cdot)$$ and arrive
at (\ref{alphacheeger}) above by applying Lemma 3.5 from \cite{BKW}.
Note that this is the bound that can also be found in \cite{Dod},
where however, Laplacians are defined in a slightly different way.

Clearly, the lower bound
$$\lambda_\Omega \geq \frac{1}{R \cdot \vol[R]},$$
from Theorem \ref{Dir-low} is stronger whenever the volume
$\vol [R]$ grows faster than linear.


\section{Including a potential} \label{potentials}

In this section we discuss how the ideas presented above allow one
in certain cases to include a potential as well.

\medskip

As usual we  assume that we are given a connected graph $(X,b,m)$
satisfying (B) and (M). We denote the set of functions $f :
X\longrightarrow \CC$ which vanish outside a finite set by $C_c
(X)$.  Let now additionally be given a bounded function
$$ V: X\longrightarrow \RR.$$
Then, we define the form $\en_V := \en + V$ and denote the
associated selfadjoint operator by $L_V$ and set
$$\lambda_V :=\inf \sigma (L_V).$$
As $V$ is bounded, so are  $\en_V$ and $L_V$.  In fact, $L_V$ acts
via
$$(L_V f) (x) = \frac{1}{m(x)} \left (\sum_{y\in X} b(x,y)( f(x) - f(y)) +
V(x) f(x)\right).$$

By general principles, sometimes discussed under the name of
Allegretto-Piepenbrinck theorem, e.g.\ \cite{HK}, there exists a a non-negative
\textit{ground state} to $L_V$ i.e. a function
$$\phi :
X\longrightarrow (0,\infty) $$ satisfying the  summability condition
$$\sum_{y\in X} b(x,y) \phi (y) <\infty$$
for every  $x\in X$ as well as
$$\frac{1}{m(x)} \left(\sum_{y\in X} b(x,y) (\phi (x) - \phi (y)) + V(x) \phi
(x)\right) - \lambda_V \phi (x)  \geq 0$$ for all $x\in
X$.\footnote{If the graph is locally finite i.e. $\#\{ y\in X:
b(x,y) >0\} <\infty$ holds for all $x\in X$, one can find $\phi$
with equality in the previous inequality.} Note that all sums in the
previous inequality are absolutely convergent due to the summability
condition satisfied by $\phi$. Then, a \textit{ground state
transform} as discussed e.g. in \cite{HK}, gives
\begin{equation}\label{eq-gs} (\en_V (f,f) -
\lambda_V (f,f)) \geq \en_\phi \left( \frac{f}{\phi},
\frac{f}{\phi}\right)
\end{equation} for all $f\in C_c (X)$. Here, $\en_\phi$ is the form associated
to the graph $(X,b_\phi, m_\phi)$ with
$$ b_\phi (x,y) = \phi (x) \phi (y) b(x,y)$$
for all $x,y\in X$ and $m_\phi (x) = \phi (x)^ 2 m(x)$ for $x\in X$.
Specifically,
$$\en_\phi (f,g) = \frac{1}{2}\sum_{x,y\in X} b_\phi (x,y) (f(x) -
f(y))(\overline{g(x)} - \overline{g(y)})$$ for $g\in C_c (X)$.

Assume now that the ground state $\phi$ is \textit{regular} i.e.
there exists a  $c\geq 1$, called the \textit{bound} on $\phi$ with
$$ 0 < \frac{1}{c} \leq \phi (x) \leq c$$
for all $x\in X$ (see \cite{FSW} for further discussion of regular
ground states).  Then, the graph $(X, b_\phi, m_\phi)$ satisfies the
assumption (B) and (M).

Moreover, the distance $d_\phi$ associated to $(X,b_\phi, m_\phi)$
is equivalent to the distance $d$ associated to $(X,b,m)$ in the
sense that we have
$$ \frac{1}{c^2}  d (x,y) \leq d_\phi (x,y) \leq c^2 d(x,y)$$
for all $x,y\in X$. Similarly, volumes are equivalent for any finite
set $S\subset X$ in the sense that we have
$$ \frac{1}{c^2} m (S)\leq m_\phi (S) \leq c^2 \vol (S).$$ Let now $D\subset X$ be relatively dense with respect to $d$ with
covering radius $R$. Set $\Omega := X\setminus D$. As $d$ and
$d_\phi$ are equivalent, then $D$ is relatively dense with covering
radius less than $c^2 R$ with respect to $d_\phi$.

We  can now apply Theorem \ref{Dir-low} to $(X,b_\phi, m_\phi)$.
Taking into acount the  equivalence of metrics and volumes we obtain
from this theorem
\begin{equation}\label{eq-estimate-en-phi}
\en_{\phi} (g,g) \geq \frac{1}{ c^4\cdot  R \cdot \vol [c^2  R]}
\|g\|^2_{\ell^2 (X,m_\phi)} \end{equation} for all $g\in C_c
(\Omega)$. We can then combine \eqref{eq-gs} and
 \eqref{eq-estimate-en-phi} to obtain
 $$(\en_V (f,f) - \lambda_V (f,f)) \geq \frac{1}{ c^4\cdot R\cdot \vol[c^2 R] }
\|f /  \phi  \|^2_{\ell^2 (X,m_\phi)} =  \frac{1}{ c^4\cdot R\cdot
\vol[c^2 R] } \|f\|^2$$ for all $f\in C_c (\Omega)$. So, if we
define $\en_{V,\Omega}$ as  the restriction of $\en_V$ to
$\ell^2(\Omega,m)$ we can summarize the preceding considerations in
the following theorem.

\begin{theorem} Let $(X,b,m)$ be as above and $V : X\longrightarrow
\RR$ bounded and  $\en_V$ the associated form. Let $D\subset X$ be
relatively dense with covering radius $R$.  If there exists a
regular ground state with bound $c$ to $\en_V$  then,
$$  \en_{V,\Omega} \geq \lambda_V  + \frac{1}{ c^4 \cdot R
\cdot \vol[c^2 R] }$$ holds.
\end{theorem}

\begin{rem}
If the metric $d$ satisfies the \textit{volume doubling property}
that there exists an $N>0$ with $\vol [\alpha s]  \leq \alpha^N \vol
[s]$ for all $\alpha \geq 1$ and $s>0$ we can further estimate the
bound in the previous theorem as
$$\en_{V,\Omega} \geq \lambda_V  + \frac{1}{ c^{4+2N} \cdot R \cdot \vol[R] }.$$
\end{rem}

\end{document}